\documentclass[11pt,oneside,english]{amsart}
\usepackage[T1]{fontenc}
\usepackage[latin9]{inputenc}
\usepackage{parskip}
\usepackage{babel}
\usepackage{amstext}
\usepackage{amsthm}
\usepackage{amssymb}
\usepackage{stmaryrd}
\usepackage[a4paper]{geometry}
\usepackage{setspace}
\usepackage[pdfusetitle,
 bookmarks=true,bookmarksnumbered=false,bookmarksopen=false,
 breaklinks=false,pdfborder={0 0 1},backref=false,colorlinks=false]
 {hyperref}

\makeatletter
\numberwithin{equation}{section}
\numberwithin{figure}{section}
\theoremstyle{plain}
\newtheorem{thm}{\protect\theoremname}[section]
\theoremstyle{definition}
\newtheorem{defn}[thm]{\protect\definitionname}
\theoremstyle{definition}
\newtheorem{example}[thm]{\protect\examplename}
\theoremstyle{remark}
\newtheorem{rem}[thm]{\protect\remarkname}
\theoremstyle{plain}
\newtheorem{prop}[thm]{\protect\propositionname}
\theoremstyle{plain}
\newtheorem{cor}[thm]{\protect\corollaryname}
\theoremstyle{plain}
\newtheorem{lem}[thm]{\protect\lemmaname}
\theoremstyle{remark}
\newtheorem*{acknowledgement*}{\protect\acknowledgementname}

\usepackage{babel}

\usepackage{ytableau}

\providecommand{\corollaryname}{Corollary}
\providecommand{\definitionname}{Definition}
\providecommand{\examplename}{Example}
\providecommand{\lemmaname}{Lemma}
\providecommand{\propositionname}{Proposition}
\providecommand{\theoremname}{Theorem}
\providecommand{\remarkname}{Remark}

\global\long\def\N{\mathbb{\mathbb{N}}}%
\global\long\def\R{\mathbb{\mathbb{R}}}%
\global\long\def\Z{\mathbb{\mathbb{Z}}}%
\global\long\def\val{\mathbb{\mathrm{val}}}%
\global\long\def\Qp{\mathbb{Q}_{p}}%
\global\long\def\Zp{\mathbb{Z}_{p}}%
\global\long\def\ac{\mathrm{ac}}%
\global\long\def\C{\mathbb{C}}%
\global\long\def\Q{\mathbb{Q}}%
\global\long\def\supp{\mathrm{supp}}%
\global\long\def\VF{\mathrm{VF}}%
\global\long\def\RF{\mathrm{RF}}%
\global\long\def\VG{\mathrm{VG}}%
\global\long\def\spec{\mathrm{Spec}}%
\global\long\def\orb{\mathrm{orb}}%
\global\long\def\rss{\mathrm{rss}}%
\global\long\def\mexp{\mathbf{e}}%
\global\long\def\GL{\mathrm{GL}}%
\global\long\def\U{\mathrm{U}}%
\global\long\def\SU{\mathrm{SU}}%
\global\long\def\Rep{\mathrm{Rep}}%
\global\long\def\Wg{\mathrm{Wg}}%
\global\long\def\Ldp{\mathcal{L}_{\mathrm{DP}}}%
\global\long\def\sgn{\mathrm{sgn}}%
\global\long\def\Irr{\mathrm{Irr}}%
\global\long\def\Ind{\mathrm{Ind}}%
\global\long\def\g{\mathfrak{g}}%
\global\long\def\l{\mathfrak{l}}%
\global\long\def\p{\mathfrak{p}}%
\global\long\def\n{\mathfrak{n}}%
\global\long\def\m{\mathfrak{m}}%
\global\long\def\t{\mathfrak{t}}%
\global\long\def\q{\mathfrak{q}}%
\global\long\def\gl{\mathfrak{gl}}%
\global\long\def\sl{\mathfrak{sl}}%
\global\long\def\ggm{\underline{\mathfrak{g}}}%
\global\long\def\llm{\underline{\mathfrak{l}}}%
\global\long\def\ppm{\underline{\mathfrak{p}}}%
\global\long\def\nnm{\underline{\mathfrak{n}}}%
\global\long\def\ttm{\underline{\mathfrak{t}}}%
\global\long\def\qqm{\underline{\mathfrak{q}}}%
\global\long\def\Lie{\mathrm{Lie}}%
\global\long\def\ad{\operatorname{ad}}%
\global\long\def\lct{\operatorname{lct}}%
\global\long\def\Hom{\operatorname{Hom}}%
\global\long\def\Ad{\operatorname{Ad}}%
\global\long\def\Int{\operatorname{Int}}%
\global\long\def\GG{\underline{G}}%
\global\long\def\PP{\underline{P}}%
\global\long\def\QQ{\underline{Q}}%
\global\long\def\LL{\underline{L}}%
\global\long\def\NN{\underline{N}}%
\global\long\def\sX{\mathsf{X}}%
\global\long\def\sY{\mathsf{Y}}%
\global\long\def\sV{\mathsf{V}}%
\global\long\def\tr{\operatorname{tr}}%

\providecommand{\acknowledgementname}{Acknowledgement}


\makeatother

\providecommand{\acknowledgementname}{Acknowledgement}
\providecommand{\corollaryname}{Corollary}
\providecommand{\definitionname}{Definition}
\providecommand{\examplename}{Example}
\providecommand{\lemmaname}{Lemma}
\providecommand{\propositionname}{Proposition}
\providecommand{\remarkname}{Remark}
\providecommand{\theoremname}{Theorem}

\begin{document}
\title[Analytic log-canonical threshold over local fields of positive characteristic]{A lower bound on the analytic log-canonical threshold over local fields
of positive characteristic}
\author{Itay Glazer}
\address{Department of Mathematics, Technion - Israel Institute of Technology,
Haifa, Israel}
\email{glazer@technion.ac.il}
\urladdr{https://sites.google.com/view/itay-glazer}
\author{Yotam I. Hendel}
\address{Department of Mathematics, Ben-Gurion University of the Negev, Be'er
Sheva 84105, Israel}
\email{yhendel@bgu.ac.il}
\urladdr{https://sites.google.com/view/yotam-hendel}
\begin{abstract}
Given a local field $F$ of positive characteristic, an $F$-analytic
manifold $X$ and an analytic function $f:X\rightarrow F$, the $F$-analytic
log-canonical threshold $\mathrm{lct}_{F}(f;x_{0})$ is the supremum
over the values $s\geq0$ such that $\left|f\right|_{F}^{-s}$ is
integrable near $x_{0}\in X$. We show that $\mathrm{lct}_{F}(f;x_{0})>0$.
Moreover, if $f$ is a regular function on a smooth algebraic $F$-variety,
we obtain an effective lower bound $\mathrm{lct}_{F}(f;x_{0})>C$,
where $C>0$ is explicit and depends only on the complexity class
of $X$ and $f$.
\end{abstract}


\maketitle
\pagenumbering{arabic}

\global\long\def\N{\mathbb{\mathbb{N}}}%
\global\long\def\R{\mathbb{\mathbb{R}}}%
\global\long\def\Z{\mathbb{\mathbb{Z}}}%
\global\long\def\val{\mathbb{\mathrm{val}}}%
\global\long\def\Qp{\mathbb{Q}_{p}}%
\global\long\def\Zp{\mathbb{Z}_{p}}%
\global\long\def\ac{\mathrm{ac}}%
\global\long\def\C{\mathbb{C}}%
\global\long\def\Q{\mathbb{Q}}%
\global\long\def\supp{\mathrm{supp}}%
\global\long\def\VF{\mathrm{VF}}%
\global\long\def\RF{\mathrm{RF}}%
\global\long\def\VG{\mathrm{VG}}%
\global\long\def\spec{\mathrm{Spec}}%
\global\long\def\orb{\mathrm{orb}}%
\global\long\def\rss{\mathrm{rss}}%
\global\long\def\mexp{\mathbf{e}}%
\global\long\def\GL{\mathrm{GL}}%
\global\long\def\U{\mathrm{U}}%
\global\long\def\SU{\mathrm{SU}}%
\global\long\def\Rep{\mathrm{Rep}}%
\global\long\def\Wg{\mathrm{Wg}}%
\global\long\def\Ldp{\mathcal{L}_{\mathrm{DP}}}%
\global\long\def\sgn{\mathrm{sgn}}%
\global\long\def\Irr{\mathrm{Irr}}%
\global\long\def\Ind{\mathrm{Ind}}%
\global\long\def\g{\mathfrak{g}}%
\global\long\def\l{\mathfrak{l}}%
\global\long\def\p{\mathfrak{p}}%
\global\long\def\n{\mathfrak{n}}%
\global\long\def\m{\mathfrak{m}}%
\global\long\def\t{\mathfrak{t}}%
\global\long\def\q{\mathfrak{q}}%
\global\long\def\gl{\mathfrak{gl}}%
\global\long\def\sl{\mathfrak{sl}}%
\global\long\def\ggm{\underline{\mathfrak{g}}}%
\global\long\def\llm{\underline{\mathfrak{l}}}%
\global\long\def\ppm{\underline{\mathfrak{p}}}%
\global\long\def\nnm{\underline{\mathfrak{n}}}%
\global\long\def\ttm{\underline{\mathfrak{t}}}%
\global\long\def\qqm{\underline{\mathfrak{q}}}%
\global\long\def\Lie{\mathrm{Lie}}%
\global\long\def\ad{\operatorname{ad}}%
\global\long\def\lct{\operatorname{lct}}%
\global\long\def\Hom{\operatorname{Hom}}%
\global\long\def\Ad{\operatorname{Ad}}%
\global\long\def\Int{\operatorname{Int}}%
\global\long\def\GG{\underline{G}}%
\global\long\def\PP{\underline{P}}%
\global\long\def\QQ{\underline{Q}}%
\global\long\def\LL{\underline{L}}%
\global\long\def\NN{\underline{N}}%
\global\long\def\sX{\mathsf{X}}%
\global\long\def\sY{\mathsf{Y}}%
\global\long\def\sV{\mathsf{V}}%
\global\long\def\tr{\operatorname{tr}}%
\raggedbottom

\section{Introduction}

Let $F$ be a non-Archimedean local field of arbitrary characteristic,
with ring of integers $\mathcal{O}_{F}$, residue cardinality $q_{F}$,
and absolute value $|\cdot|_{F}.$ Let $X$ be an $F$-analytic manifold
of dimension $n$. Let $(U_{\alpha}\subset X,\psi_{\alpha}:U_{\alpha}\to F^{n})_{\alpha\in\mathcal{I}}$
be an atlas, and fix a Haar measure $\mu_{F^{n}}$ on $F^{n}$ such
that the volume of $\mathcal{O}_{F}^{n}$ is $1$. Let $C^{\infty}(X)$
be the space of smooth (i.e.~locally constant) functions on $X$,
and let $C_{c}^{\infty}(X)$ be the subspace of smooth compactly supported
functions. Let $\mathcal{M}^{\infty}(X)$ be the space of smooth measures
on $X$, i.e.~measures such that each $(\psi_{\alpha})_{*}(\mu|_{U_{\alpha}})$
has a locally constant density with respect to the Haar measure on
$F^{n}$. Let $\mathcal{M}_{c}^{\infty}(X)$ be the subspace of compactly
supported smooth measures.

In this note we study the following invariant: 
\begin{defn}[see e.g.~\cite{Kol,Mus12}, for $F=\C$]
\label{def:analytic lct}Let $X$ be an $F$-analytic manifold, let
$x_{0}\in X$ and let $J$ be an ideal in the ring of $F$-analytic
functions on $X$. Suppose that the germ $J_{x_{0}}$ at $x_{0}$
is non-zero, and is generated by the germs of $F$-analytic functions
$f_{1},\ldots,f_{r}:V\to F$, defined on some open neighborhood $V$
of $x_{0}$. We define the $F$\textit{-analytic log-canonical threshold}
of $J$ at $x_{0}$ by 
\[
\lct_{F}(J;x_{0}):=\sup\left\{ s\ge0:\begin{array}{l}
\text{there exists an open neighborhood }U\subset V\text{ of }x_{0}\text{ such that}\\
\text{for every }\mu\in\mathcal{M}_{c}^{\infty}(U),{\displaystyle \int_{U}\min_{1\le i\le r}|f_{i}(x)|_{F}^{-s}\,d\mu(x)<\infty}
\end{array}\right\} .
\]
\end{defn}

This definition is independent of the choice of generators of $J_{x_{0}}$.
Indeed, let $g_{1},\ldots,g_{r'}$ be another collection of analytic
functions, defined on a possibly smaller neighborhood of $x_{0}$,
whose germs generate $J_{x_{0}}$. After shrinking to a sufficiently
small open compact neighborhood $B$ of $x_{0}$, we may write $f_{i}(x)=\sum_{j=1}^{r'}c_{ij}(x)g_{j}(x)$
for analytic functions $c_{ij}$ on $B$. Taking $M_{1}:=\underset{1\le i\le r,\,1\le j\le r',\,x\in B}{\max}\left|c_{ij}(x)\right|_{F}$,
for every $s>0$, we get: 
\[
\min\limits_{1\le i\le r}\left|f_{i}(x)\right|_{F}^{-s}=\min\limits_{1\le i\le r}\left|\sum_{j=1}^{r'}c_{ij}(x)g_{j}(x)\right|_{F}^{-s}\geq M_{1}^{-s}\min\limits_{1\le j\le r'}\left|g_{j}(x)\right|_{F}^{-s}.
\]
Similarly, we deduce that $\min\limits_{1\le j\le r'}\left|g_{j}(x)\right|_{F}^{-s}\geq M_{2}^{-s}\min\limits_{1\le i\le r}\left|f_{i}(x)\right|_{F}^{-s}$
for some $M_{2}>0$, so that $\lct_{F}(J;x_{0})$ is well defined.

We now state the main results of the paper. In the theorems below
we take $F$ to be a non-Archimedean local field. 

Our first main result is as follows. 
\begin{thm}
\label{thm:lower bound on lct}Let $X$ be an $F$-analytic manifold,
let $x_{0}\in X$ and let $f_{1},\ldots,f_{r}:X\to F$ be $F$-analytic
functions generating an ideal $J$ whose germ $J_{x_{0}}$ at $x_{0}$
is non-zero. Then there exists $\epsilon_{x_{0}}>0$ such that: 
\[
\lct_{F}(J;x_{0})>\epsilon_{x_{0}}.
\]
\end{thm}

The lower bound in Theorem \ref{thm:lower bound on lct} can be easily
made uniform in $x_{0}$ when $X$ is compact. If $X$ is non-compact,
such a uniform lower bound is false as one can glue together an infinite
sequence of disjoint open compact subsets, along with analytic functions
on each subset which become increasingly more singular.

In our second main result, we show that in the algebro-geometric setting,
the compactness assumption can be omitted, and a uniform lower bound
on $\lct_{F}(J;x_{0})$ can be given. 
\begin{thm}
\label{thm:uniform lower bound}Let $X$ be a smooth $F$-variety
and let $\varphi_{1},\ldots,\varphi_{r}:X\rightarrow\mathbb{A}_{F}^{1}$
be regular functions on $X$. Suppose that for every irreducible component
$X'$ of $X$, the restrictions $\varphi_{1}|_{X'},\ldots,\varphi_{r}|_{X'}$
are not all identically zero, and write $J:=\langle\varphi_{1,F},\ldots,\varphi_{r,F}\rangle$,
where $\varphi_{i,F}:X(F)\rightarrow F$ is the analytic map induced
from $\varphi_{i}$. Then there exists $\epsilon>0$, such that for
every $x_{0}\in X(F)$, 
\[
\lct_{F}(J;x_{0})>\epsilon.
\]
\end{thm}

In fact, concrete lower bounds which depend only on the complexity\footnote{For the precise definition of complexity, see e.g.~\cite[Definition 7.7]{GH19}}
of $X$ and the degrees of $\varphi_{1},\ldots,\varphi_{r}$, can
be given. 
\begin{thm}
\label{thm:concrete lower bound}Let $X\subseteq\mathbb{A}_{F}^{n+m}$
be a smooth $n$-dimensional $F$-subvariety defined by $g_{1}=\ldots=g_{m}=0$
for polynomials $g_{1},\ldots,g_{m}\in F[x_{1},\ldots,x_{n+m}]$ of
degree $\leq D$. Let $\varphi_{1},\ldots,\varphi_{r}\in F[x_{1},\ldots,x_{n+m}]$
be polynomials of degree $\leq d$, such that for every irreducible
component $X'$ of $X$, the restrictions $\varphi_{1}|_{X'},\ldots,\varphi_{r}|_{X'}$
are not all identically zero, and let $J:=\langle\varphi_{1,F},\ldots,\varphi_{r,F}\rangle$.
Then for every $x_{0}\in X(F)$: 
\[
\lct_{F}(J;x_{0})\geq\frac{1}{d\cdot D^{m}}.
\]
\end{thm}

The lower bound in Theorem \ref{thm:concrete lower bound} is optimal,
as the following example shows. 
\begin{example}
Let $d,D\geq1$. Let $X\subseteq\mathbb{A}_{F}^{1+m}$ be the smooth
curve defined by $x_{i+1}=x_{i}^{D}$ for $i=1,\ldots,m$. Let $\varphi(x_{1},\ldots,x_{m+1})=x_{m+1}^{d}$.
The projection $(x_{1},\ldots,x_{m+1})\mapsto x_{1}$ is a diffeomorphism
$X(\mathcal{O}_{F})\rightarrow\mathcal{O}_{F}$ with inverse $\psi(y)=(y,y^{D},y^{D^{2}},\ldots,y^{D^{m}})$.
Thus, 
\[
\lct_{F}(\varphi;0)=\lct_{F}(\varphi\circ\psi;0)=\lct_{F}(y^{D^{m}d};0)=\frac{1}{d\cdot D^{m}}.
\]
\end{example}

\begin{rem}
In characteristic zero, analogues of Theorems \ref{thm:lower bound on lct}
and \ref{thm:uniform lower bound} are known for arbitrary local fields,
i.e.~for $F=\R,\C$, or for a finite extension of $\Qp$, as consequences
of Hironaka's theorem on resolution of singularities \cite{Hir64}.
See, for example, \cite[Remark 1 and Corollary 2]{LCVZ13} for the
Archimedean case of analytic mappings, \cite{Kol}, \cite[Theorem 1.2]{Mus12}
and \cite{Saia} for the cases $F=\mathbb{R},\mathbb{C}$, and \cite[Theorem 2.7]{VZG08}
for the case where $F$ is a finite extension of $\Qp$. Thus, the
main point of the present paper is the fixed positive-characteristic
case. The proof of Theorem \ref{thm:concrete lower bound} should
extend to the Archimedean setting. We work throughout in the non-Archimedean
setting in order to keep the notation and the arguments uniform.
\end{rem}

The results of this paper are utilized in a series of works of Aizenbud--Gourevitch--Kazhdan--Sayag
\cite{AGKSb,AGKSc} about Harish-Chandra's regularity theorem over
local fields of positive characteristic. In a sequel to this paper,
appearing in \cite[Appendix A]{AGKSc}, we use the lower bounds above
on the $\lct_{F}$ to study integrability of pushforwards $f_{*}\mu$
of smooth, compactly supported measures $\mu$, by analytic maps $f:X\rightarrow Y$
between $F$-analytic manifolds, in positive characteristic. This
generalizes results from \cite{GH21,GHS} in characteristic $0$,
which were used in \cite{GGH} to study integrability of Harish-Chandra
characters of reductive groups over local fields.

\subsection{Discussion of the $\protect\lct_{F}$ and the strategy of proof}

When $F=\C$, the $F$-analytic log-canonical threshold is equivalent
to the \emph{log-canonical threshold} $\lct$ (also called the \emph{complex
singularity exponent} \cite[Chapter 7]{AGV88} and \cite{Var83})
which is an important singularity invariant, widely used in birational
geometry (see e.g.~\cite{Mus12,Kol}). For local fields $F$ of characteristic
$0$, the $F$-analytic log-canonical threshold was extensively studied,
especially in the context of Igusa's local zeta function \cite{Igu78,Den87,Den91a,Igu00,Saia,VZG08,Bud12,Meu16}.
In this setting, one can use a rich set of tools, such as Hironaka's
(analytic) resolution of singularities \cite{Hir64}, the theory of
motivic integration and uniform $p$-adic integration (e.g.~\cite{HK06,CL08,CL10}),
and the theory of $D$-modules \cite{Ber72}.

Some of the (geometric and model-theoretic) tools extend to local
fields of large positive characteristic. However, if $F$ is a fixed
local field of (small) positive characteristic, the above methods
break; the model theory of such local fields is not well understood\footnote{See e.g.~\cite[Remark 2]{Clu01} and \cite{CH18} for a short discussion
of some of the challenges.}, and in the algebro-geometric realm, the absence of Hironaka's theorem
in positive characteristic, along with issues relating to inseparability,
motivates geometers to consider alternative notions, such as the \emph{$\mathsf{F}$-pure
threshold }(see e.g.~\cite{TW04,BMS09} and \cite[Section 2]{Mus12}). 
\begin{rem}
In the polynomial (algebro-geometric) setting, the $\mathsf{F}$-pure
threshold and the $F$-analytic log-canonical threshold share many
similar properties (see e.g.~\cite[Section 2.2]{Mus12}), and in
particular they agree for monomials, and for other families of polynomials.
On the other hand, they are different invariants; for example, the
$\mathsf{F}$-pure threshold of $f(x,y)=x^{2}+y^{3}$ at $0$ is $\frac{5}{6}-\frac{1}{6p}$
whenever the characteristic is $p=2\,(\mathrm{mod}\,3)$ and $p>3$,
while $\lct_{F}(f)=\frac{5}{6}$ for every local field $F$ of large
enough residual characteristic. Moreover, the $\mathsf{F}$-pure threshold
is stable under finite field extensions \cite[Proposition 2.4]{BHMM12},
unlike the $\lct_{F}$ which might change. Indeed, one can consider
$f(x,y,z)=x^{2}+y^{2}+z^{2}$ for $F=\R$, and $f(x,y,z)=x^{2}+y^{2}+tz^{2}$
for $F=\mathbb{F}_{p}((t))$, where $p=3\,(\mathrm{mod}\,4)$ is a
prime. In both cases, one has $\lct_{F}(f)=\frac{3}{2}$, but $\lct_{F'}(f)=1$
for some quadratic extension $F'=\C$ or $F'=\mathbb{F}_{p}((\sqrt{t}))$.
This can be shown by changing variables using a blow-up along the
origin, which is a resolution of singularities in these cases: over
$F$, the corresponding quadratic form $f$ is anisotropic, so the
strict transform has no $F$-points above the origin and the exceptional
divisor gives $\lct_{F}(f)=\frac{3}{2}$; over $F'$, the form becomes
isotropic, so the strict transform has $F'$-points above the origin,
giving $\lct_{F'}(f)=1$ (see e.g.~\cite[Proposition 3.3]{Saia}
for $F=\R$). Finally, it would be interesting to compare these two
invariants in the polynomial setting, for a fixed local field of positive
characteristic.
\end{rem}

The strategy to prove Theorem \ref{thm:lower bound on lct} is as
follows. Firstly, we reduce to a single analytic function $f:X\to F$
(i.e.~$r=1$) and then reduce, using a Weierstrass preparation theorem
(Theorem \ref{thm:Weierstrass preparation}), to the case that $f$
is a \emph{Weierstrass polynomial}, i.e.~a monic polynomial in the
last variable $x_{n}$, whose coefficients are (strictly) convergent
power series in $(x_{1},\ldots,x_{n-1})$ (see Definition \ref{def:formal power series}
below). Secondly, we prove a Remez-type small ball estimate for monic
polynomials in one variable (see Lemma \ref{lem:Remez for monic}).
We then generalize this to a small ball estimate for a Weierstrass
polynomial (Lemma \ref{lem:small ball estimates for Weierstrass}).
Finally, we show that small ball estimates for $f:X\subseteq\mathcal{O}_{F}^{n}\rightarrow F$
imply a lower bound on $\lct_{F}(f;0)$, thus concluding the proof.

To prove Theorems \ref{thm:uniform lower bound} and \ref{thm:concrete lower bound},
we control the order of vanishing of each regular function $\varphi_{i}$
along $X$ at $x_{0}\in X(F)$ (see Lemma \ref{lem:bounds on multiplicity}),
and combine it with an effective version of the Weierstrass preparation
theorem to obtain effective bounds on $\lct_{F}(\varphi_{i};0)$ (Lemma
\ref{lem:lct of low multiplicity}).

\section{Proof of Theorem \ref{thm:lower bound on lct}}

Throughout this section $F$ is a non-Archimedean local field. We
start with the following definition. 
\begin{defn}[{\cite[Sections 5.1.1, 5.2.1-5.2.4]{BGR84}}]
~\label{def:formal power series} 
\begin{enumerate}
\item We denote by $F\left\llbracket x_{1},\ldots,x_{n}\right\rrbracket $
the ring of formal power series over $F$. 
\item The \textbf{Tate algebra} $T_{n}$ in $n$ variables is the ring of
\textbf{strictly convergent power series}: 
\[
T_{n}=F\langle x_{1},\ldots,x_{n}\rangle:=\left\{ \sum_{I\in\Z_{\geq0}^{n}}a_{I}x^{I}\in F\left\llbracket x_{1},\ldots,x_{n}\right\rrbracket :a_{I}\underset{\left|I\right|\rightarrow\infty}{\longrightarrow}0\right\} .
\]
\item The \textbf{Gauss norm} $\left\Vert \:\right\Vert _{{\tiny \mathrm{Gauss}}}$
on $T_{n}$ is defined as 
\[
\left\Vert \sum_{I\in\Z_{\geq0}^{n}}a_{I}x^{I}\right\Vert _{\mathrm{Gauss}}:=\max_{I\in\Z_{\geq0}^{n}}\left\{ \left|a_{I}\right|_{F}\right\} .
\]
\item An element $f=\sum_{s=0}^{\infty}a_{s}(x_{1},\ldots,x_{n-1})x_{n}^{s}\in T_{n}$
is said to be \textbf{distinguished in $x_{n}$ of order }$s_{0}\in\Z_{\geq0}$
if: 
\begin{enumerate}
\item The element $a_{s_{0}}$ is a unit in $T_{n-1}$. 
\item We have $\left\Vert f\right\Vert _{\mathrm{Gauss}}=\left\Vert a_{s_{0}}\right\Vert _{\mathrm{Gauss}}$
and $\left\Vert a_{s_{0}}\right\Vert _{\mathrm{Gauss}}>\left\Vert a_{s}\right\Vert _{\mathrm{Gauss}}$
for $s>s_{0}$. 
\end{enumerate}
\item A monic polynomial $f\in T_{n-1}[x_{n}]$ in the variable $x_{n}$,
with $\left\Vert f\right\Vert _{\mathrm{Gauss}}=1$ is called a \textbf{Weierstrass
polynomial}. 
\end{enumerate}
~ 
\end{defn}

\begin{prop}[{\cite[Section 5.2.4, Proposition 1]{BGR84}}]
\label{prop: unipotent automorphism of F^n}Let $0\neq f\in T_{n}$.
Then there exists an $F$-analytic diffeomorphism $\phi:\mathcal{O}_{F}^{n}\to\mathcal{O}_{F}^{n}$
of the form 
\begin{equation}
\phi(x_{1},\ldots,x_{n})=(x_{1}+x_{n}^{d_{1}},\ldots,x_{n-1}+x_{n}^{d_{n-1}},x_{n}),\label{eq:simple automorphism}
\end{equation}
for $d_{1},\ldots,d_{n-1}\in\N$, such that $f\circ\phi$ is distinguished
in $x_{n}$ (of some order $s_{0}$). 
\end{prop}

\begin{thm}[{Weierstrass preparation theorem, \cite[Section 5.2.2, Theorem 1]{BGR84}}]
\label{thm:Weierstrass preparation}Let $f(x_{1},\ldots,x_{n})\in T_{n}$
be distinguished of order $s_{0}$. Then there exist a unit $h(x_{1},\ldots,x_{n})\in T_{n}$
and elements $r_{0},r_{1},\ldots,r_{s_{0}-1}\in T_{n-1}$ such that
\begin{equation}
f(x_{1},\ldots,x_{n})\cdot h(x_{1},\ldots,x_{n})=x_{n}^{s_{0}}+\sum_{i=0}^{s_{0}-1}r_{i}(x_{1},\ldots,x_{n-1})x_{n}^{i}\in T_{n-1}[x_{n}]\label{eq:Weierstrass}
\end{equation}
is a Weierstrass polynomial. 
\end{thm}

\begin{cor}
\label{cor:reduction to Weierstrass polynomial}In order to prove
Theorem \ref{thm:lower bound on lct}, we may assume that $r=1$,
$X=\mathcal{O}_{F}^{n}$, $x_{0}=0$, and moreover that $f_{1}(x_{1},\ldots,x_{n})\in T_{n}$
is a Weierstrass polynomial. 
\end{cor}

\begin{proof}
Since the germ $J_{x_{0}}$ at $x_{0}$ is non-zero, without loss
of generality, we may assume that the germ of $f_{1}$ at $x_{0}$
is non-zero. By Definition \ref{def:analytic lct}, we have $\lct_{F}(J;x_{0})\geq\lct_{F}(f_{1};x_{0})$.
Thus, we may assume $r=1$. We choose a diffeomorphism $\psi:U\rightarrow\mathcal{O}_{F}^{n}$,
with $x_{0}\in U\subseteq X$, and $\psi(x_{0})=0$. Writing $f:=f_{1}\circ\psi^{-1}$,
it is enough to prove that $\lct_{F}(f;0)>0$.

Without loss of generality, we may assume that the valuation in $F$
takes values in $\Z$, and let $\varpi_{F}$ be a uniformizer of $F$,
so that $\left|\varpi_{F}\right|_{F}=q_{F}^{-1}$. Since $f$ is analytic,
taking $k\in\N$ large enough, $f|_{\varpi_{F}^{k}\mathcal{O}_{F}^{n}}$
is given by a converging power series. So $\widetilde{f}:=f(\varpi_{F}^{k}x_{1},\ldots,\varpi_{F}^{k}x_{n})$
converges in $\mathcal{O}_{F}^{n}$, and thus $\widetilde{f}\in T_{n}$
(see e.g.~\cite[Section 5.1.4, Proposition 1]{BGR84}). Since $\lct_{F}(f;0)=\lct_{F}(\widetilde{f};0)$
we may also assume $f=\widetilde{f}$. Since the map $\phi$ from
(\ref{eq:simple automorphism}) is a diffeomorphism of $\mathcal{O}_{F}^{n}$,
we may assume that $f$ is distinguished in $x_{n}$. By Theorem \ref{thm:Weierstrass preparation},
we may replace $f$ with $f\cdot h$, where $h\in T_{n}$ is invertible,
obtaining a Weierstrass polynomial. 
\end{proof}
We now turn to obtain small ball estimates for monic polynomials and
Weierstrass polynomials. We first need the following elementary lemma. 
\begin{lem}
\label{lem:elementary lemma}Let $F$ be a non-Archimedean local field,
and let $F'/F$ be a finite field extension, with the norm on $F'$
chosen such that it extends the norm on $F$. Let $B(x',r)\subseteq F'$
be a ball in $F'$. Then either $B(x',r)\cap F=\varnothing$, or for
any $x\in B(x',r)\cap F$, 
\[
B(x',r)\cap F=\{y\in F:|x-y|_{F}\leq r\}.
\]
\end{lem}

\begin{proof}
Let $B(x',r)\subseteq F'$ be a ball as above. Recall that for any
$x\in B(x',r)$ we have $B(x',r)=B(x,r)$. If $B(x',r)\cap F\neq\varnothing$,
then for any $x\in B(x',r)\cap F$ we have $B(x',r)\cap F=B(x,r)\cap F$.
The desired equality now follows as the norm on $F'$ extends the
norm on $F$. 
\end{proof}
\begin{lem}
\label{lem:Remez for monic}Let $f(x)=x^{d}+a_{d-1}x^{d-1}+\ldots+a_{1}x+a_{0}$
be a monic polynomial, with $a_{i}\in F$ for $0\leq i\leq d-1$.
Then for every $\delta>0$, we have: 
\[
\mu_{F}\left(\left\{ x\in F:\left|f(x)\right|_{F}\leq\delta\right\} \right)\le d\delta^{\frac{1}{d}}.
\]
\end{lem}

\begin{proof}
Note there exists a finite extension $F'$ of $F$ (extending the
norm on $F$), such that $f$ breaks into a product of linear terms
\[
f(x)=(x-\alpha_{1})\cdot\ldots\cdot(x-\alpha_{d}),\text{ with }\alpha_{j}\in F'.
\]
Hence, for every $\delta>0$ we have, 
\begin{align*}
\left\{ x\in F:\left|f(x)\right|_{F}\leq\delta\right\}  & =F\cap\left\{ x'\in F':\left|(x'-\alpha_{1})\cdot\ldots\cdot(x'-\alpha_{d})\right|_{F'}\leq\delta\right\} \\
 & \subseteq F\cap\left\{ x'\in F':\exists1\leq j\leq d\text{ s.t.\,}\left|(x'-\alpha_{j})\right|_{F'}\leq\delta^{\frac{1}{d}}\right\} .
\end{align*}
Note that $\left\{ x'\in F':\left|(x'-\alpha_{j})\right|_{F'}\leq\delta^{\frac{1}{d}}\right\} $
is a ball $B'_{j}:=B(\alpha_{j},\delta^{\frac{1}{d}})$ in $F'$.
By Lemma \ref{lem:elementary lemma}, $B'_{j}\cap F$ is either empty,
or a ball $B_{j}$ of radius $\delta^{\frac{1}{d}}$ in $F$. We therefore
get
\[
\mu_{F}\left(\left\{ x\in F:\left|f(x)\right|_{F}\leq\delta\right\} \right)\leq\sum_{j=1}^{d}\mu_{F}(B(\alpha_{j},\delta^{\frac{1}{d}})\cap F)\leq d\delta^{\frac{1}{d}}.\qedhere
\]
\end{proof}
\begin{lem}
\label{lem:small ball estimates for Weierstrass}Let $f\in T_{n-1}[x_{n}]$
be a Weierstrass polynomial of degree $d$. Then $\forall\delta>0$,
\[
\mu_{F^{n}}\left(\left\{ x\in\mathcal{O}_{F}^{n}:\left|f(x)\right|_{F}\leq\delta\right\} \right)\le d\delta^{\frac{1}{d}}.
\]
\end{lem}

\begin{proof}
For each $z=(z_{1},\ldots,z_{n-1})\in\mathcal{O}_{F}^{n-1}$, set
$g_{z}(x_{n}):=f(z_{1},\ldots,z_{n-1},x_{n})$. Since $g_{z}$ is
a monic polynomial in $x_{n}$ of degree $d$, we get by Lemma \ref{lem:Remez for monic},
\begin{align*}
 & \mu_{F^{n}}\left(\left\{ (z,x_{n})\in\mathcal{O}_{F}^{n}:\left|f(z,x_{n})\right|_{F}\leq\delta\right\} \right)\\
= & \int_{\mathcal{O}_{F}^{n-1}}\left(\int_{\mathcal{O}_{F}}1_{\left\{ x_{n}\in\mathcal{O}_{F}:\left|g_{z}(x_{n})\right|_{F}\leq\delta\right\} }dx_{n}\right)dz_{1}\ldots dz_{n-1}\leq\int_{\mathcal{O}_{F}^{n-1}}d\delta^{\frac{1}{d}}dz_{1}\dots dz_{n-1}\leq d\delta^{\frac{1}{d}}.\qedhere
\end{align*}
\end{proof}
\begin{rem}
Estimates as in Lemmas \ref{lem:Remez for monic} and \ref{lem:small ball estimates for Weierstrass}
are called \emph{small ball} (or \emph{sublevel}) estimates. The case
of $F=\R$ and $n=1$ goes back to P\'olya and Remez \cite{Rem36}
in the 30's. Their works imply that for a monic degree $d$ polynomial
$f\in\R[x]$:
\[
\mu_{\mathrm{Lebesgue}}\left\{ x\in\R:\left|f(x)\right|<\delta\right\} <4\delta^{\frac{1}{d}}.
\]
Cubarikov \cite[Lemma 2]{Cub76} obtained small ball estimates for
$F=\Qp$. For $n>1$, various small ball estimates for polynomial
maps were obtained (see e.g.~\cite[Theorem 7.1]{CCW99}, \cite[Theorem 8]{carbery2001distributional}
and the discussion in \cite[Sections 2.3-2.5]{GM22}). 
\end{rem}

We can now finish the proof of Theorem \ref{thm:lower bound on lct}: 
\begin{proof}[Proof of Theorem \ref{thm:lower bound on lct}]
By Corollary \ref{cor:reduction to Weierstrass polynomial}, we may
assume that $r=1$, $x_{0}=0$, and that $f(x)$ is a Weierstrass
polynomial of degree $d$. If $d=0$, then $f$ is a unit and there
is nothing to prove. Hence assume $d\geq1$. By Lemma \ref{lem:small ball estimates for Weierstrass},
for every $s<\frac{1}{d}$, 
\begin{align*}
\int_{\mathcal{O}_{F}^{n}}\left|f(x)\right|_{F}^{-s}\mu_{F^{n}} & =\sum_{k=0}^{\infty}\mu_{F^{n}}\left(\left\{ x\in\mathcal{O}_{F}^{n}:\val(f(x))=k\right\} \right)q_{F}^{ks}\\
 & \leq\sum_{k=0}^{\infty}\mu_{F^{n}}\left(\left\{ x\in\mathcal{O}_{F}^{n}:\val(f(x))\geq k\right\} \right)q_{F}^{ks}\\
 & \leq\sum_{k=0}^{\infty}dq_{F}^{-k/d}\cdot q_{F}^{ks}=d\sum_{k=0}^{\infty}q_{F}^{-k(\frac{1}{d}-s)}=\frac{d}{1-q_{F}^{-(\frac{1}{d}-s)}}<\infty,
\end{align*}
which implies $\lct_{F}(f;0)\geq\frac{1}{d}$. In particular, Theorem
\ref{thm:lower bound on lct} follows by taking any $0<\epsilon_{x_{0}}<1/d$.
\end{proof}

\subsection{Proof of Theorems \ref{thm:uniform lower bound} and \ref{thm:concrete lower bound}}

Theorem \ref{thm:uniform lower bound} is a consequence of Theorem
\ref{thm:concrete lower bound}, since each smooth $F$-variety $X$
is locally (in the Zariski topology) of the form as in Theorem \ref{thm:concrete lower bound}.

The proof of Theorem \ref{thm:concrete lower bound} is based on Lemmas
\ref{lem:bounds on multiplicity} and \ref{lem:lct of low multiplicity}
below. 
\begin{defn}[{\cite[Definition 6.7.7]{HS06}}]
\label{def:order of vanishing}Let $X$ be an $F$-variety, $x_{0}\in X(F)$,
and $\varphi$ a regular function on $X$. The \emph{order} (or \emph{multiplicity})
of $\varphi$ at $x_{0}$, along $X$ is 
\[
\mathrm{ord}_{X,x_{0}}(\varphi)=\sup\left\{ k\in\N:\varphi\in\mathfrak{m}_{x_{0}}^{k}\right\} ,
\]
where $\mathfrak{m}_{x_{0}}$ is the maximal ideal in the local ring
$\mathcal{O}_{X,x_{0}}$. 
\end{defn}

\begin{lem}
\label{lem:bounds on multiplicity}In the setting of Theorem \ref{thm:concrete lower bound},
let $x_{0}\in X(F)$, and let $\varphi\in F[x_{1},\ldots,x_{n+m}]$
be a polynomial of degree $d$ which doesn't vanish identically on
the irreducible component $X_{0}$ containing $x_{0}$. Then 
\[
\mathrm{ord}_{X,x_{0}}(\varphi)\leq d\cdot D^{m}.
\]
\end{lem}

\begin{proof}
Since $X$ is a smooth $F$-variety, then $(\mathcal{O}_{X,x_{0}},\mathfrak{m}_{x_{0}})$
is a regular local ring. If $\varphi(x_{0})\neq0$, then $\mathrm{ord}_{X,x_{0}}(\varphi)=0\leq d\cdot D^{m}$.
Suppose that $\varphi(x_{0})=0$, and let $Z$ be the scheme $Z:=X_{0}\cap\left\{ \varphi=0\right\} $.
By \cite[Example 11.2.8]{HS06}, we have $\mathrm{ord}_{X,x_{0}}(\varphi)=e_{\mathrm{HS}}(Z,x_{0})$,
where $e_{\mathrm{HS}}$ denotes the \emph{Hilbert--Samuel multiplicity}
of $Z$ at $x_{0}\in Z(F)$ (see \cite[Definition 11.1.5]{HS06}).
By Bezout's theorem, 
\[
e_{\mathrm{HS}}(Z,x_{0})\leq\deg(\varphi)\cdot\prod_{j=1}^{m}\deg(g_{j})\leq d\cdot D^{m}.
\]
This concludes the proof.
\end{proof}
If $f:U\rightarrow F$ is an analytic function in an open neighborhood
$U$ of $y\in F^{n}$, we write
\[
\mathrm{ord}_{y}(f):=\max\left\{ k\in\mathbb{Z}_{\ge0}:f\in(\mathfrak{m}_{y}^{\mathrm{an}})^{k}\right\} ,
\]
where $\mathfrak{m}_{y}^{\mathrm{an}}$ is the maximal ideal of analytic
germs vanishing at $y$ in the ring $\mathcal{O}_{F^{n},y}^{\mathrm{an}}$
of analytic germs at $y$. In particular, if $y=0$, and $f(x_{1},\ldots,x_{n})=\sum_{I\in\mathbb{Z}_{\ge0}^{n}}a_{I}x^{I}$
is the convergent power series expansion of $f$ near $0$, then 
\[
\mathrm{ord}_{0}(f)=\min\left\{ \left|I\right|:a_{I}\neq0\right\} ,
\]
and moreover, $\mathrm{ord}_{0}(f)=0$ if and only if $f(0)\neq0$.

In addition, for a power series $f(x_{1},\ldots,x_{n})=\sum_{I\in\mathbb{Z}_{\ge0}^{n}}a_{I}x^{I}$,
and $M\in\N$, we denote by $f_{M}:=\sum_{I:\left|I\right|=M}a_{I}x^{I}$
its $M$-th homogeneous part.
\begin{lem}
\label{lem:lct of low multiplicity}Let $U\subseteq F^{n}$ be an
open compact subset, and $x_{0}\in U$, and let $f:U\rightarrow F$
be an analytic function, with $K:=\mathrm{ord}_{x_{0}}(f)$. Then
$\lct_{F}(f;x_{0})\geq\frac{1}{K}$ if $K>0$, and $\lct_{F}(f;x_{0})=\infty$
if $K=0$. 
\end{lem}

\begin{proof}
We may assume that $x_{0}=0$. If $K=0$, then $f(0)\neq0$. Hence,
$\left|f\right|_{F}^{-s}$ is locally bounded near $0$ for every
$s\ge0$, and therefore $\lct_{F}(f;0)=\infty$. We may thus assume
that $K\ge1$.

By shrinking $U$ around $0$, we may assume that $f$ is given by
a converging power series 
\[
f(x_{1},\ldots,x_{n})=\sum_{I\in\Z_{\geq0}^{n}}a_{I}x^{I}=\sum_{k=K}^{\infty}f_{k},\text{ where }f_{k}:=\sum_{I\in\Z_{\geq0}^{n}:\left|I\right|=k}a_{I}x^{I}.
\]
Hence, there exist $a\in\N$ and $b\in\Z$ such that the function
$g(x_{1},\ldots,x_{n}):=\varpi_{F}^{-b}f(\varpi_{F}^{a}x_{1},\ldots,\varpi_{F}^{a}x_{n})$
satisfies: 
\begin{enumerate}
\item $g=\sum_{I\in\Z_{\geq0}^{n}}b_{I}x^{I}\in T_{n}$.
\item $\underset{I:\left|I\right|=K}{\sup}\left|b_{I}\right|_{F}=1$. 
\item $\underset{I:\left|I\right|>K}{\sup}\left|b_{I}\right|_{F}<1$, so
that $\left\Vert g\right\Vert _{\mathrm{Gauss}}=1$. 
\end{enumerate}
Since $\lct_{F}(f;0)=\lct_{F}(g;0)$, we may replace $f$ with $g$.
We now make a linear change of variables $T$ so that the coefficient
of $x_{n}^{K}$ in $(g\circ T)_{K}$ is non-zero. If $n=1$, this
is automatic. Assume $n>1$. Consider the polynomial 
\[
P(c_{1},\ldots,c_{n-1}):=g_{K}(c_{1},\ldots,c_{n-1},1).
\]
Since $g_{K}$ is a non-zero homogeneous polynomial, $P$ is a non-zero
polynomial. Since $\left(\mathcal{O}_{F}^{\times}\right)^{n-1}\subseteq F^{n-1}$
is an open subset, and $\left\{ P=0\right\} $ is either empty, or
a codimension $1$ subset in $F^{n-1}$, we may fix $c_{1},\ldots,c_{n-1}\in\mathcal{O}_{F}^{\times}$
such that $P(c_{1},\ldots,c_{n-1})\neq0$. In particular, $\left|P(c_{1},\ldots,c_{n-1})\right|_{F}=q_{F}^{-k_{0}}$
for some $k_{0}\in\Z_{\geq0}$. We define $T:\mathcal{O}_{F}^{n}\rightarrow\mathcal{O}_{F}^{n}$,
by: 
\[
T(x_{1},\ldots,x_{n}):=(x_{1}+c_{1}x_{n},\ldots,x_{n-1}+c_{n-1}x_{n},x_{n}),
\]
and get that 
\[
g\circ T(x_{1},\ldots,x_{n})=\sum_{I:\left|I\right|=K}b_{I}\cdot(x^{I}\circ T)+\sum_{I:\left|I\right|>K}b_{I}\cdot(x^{I}\circ T).
\]
Therefore, the coefficient of $x_{n}^{K}$ in $(g\circ T)_{K}=g_{K}\circ T(x_{1},\ldots,x_{n})$
is $g_{K}(c_{1},\ldots,c_{n-1},1)\neq0$, of absolute value $q_{F}^{-k_{0}}$. 

We now rescale $g\circ T$ in two steps. First, define
\[
g_{1}(x_{1},\ldots,x_{n}):=\varpi_{F}^{-k_{0}-K(k_{0}+1)}(g\circ T)(\varpi_{F}^{k_{0}+1}x_{1},\ldots,\varpi_{F}^{k_{0}+1}x_{n}).
\]
Write $g_{1}=\sum_{\left|I\right|\geq K}c_{I}x^{I}$. Then the coefficient
of $x_{n}^{K}$ in $g_{1}$ has absolute value $1$. Moreover, $\max_{|I|=K}\left|c_{I}\right|_{F}\leq q_{F}^{k_{0}}$,
while $\left|c_{I}\right|_{F}\leq q_{F}^{-1}$ for every $I$ with
$\left|I\right|>K$. Next, we define
\[
\widetilde{g}(x_{1},\ldots,x_{n}):=g_{1}(\varpi_{F}^{k_{0}+1}x_{1},\ldots,\varpi_{F}^{k_{0}+1}x_{n-1},x_{n}).
\]
Write $\widetilde{g}=\sum_{|I|\geq K}d_{I}x^{I}$. The coefficient
of $x_{n}^{K}$ still has absolute value $1$. If $I\neq(0,\ldots,0,K)$
and $\left|I\right|=K$, then $I_{1}+\cdots+I_{n-1}\geq1$, hence
$\left|d_{I}\right|_{F}\leq q_{F}^{k_{0}}q_{F}^{-(k_{0}+1)}=q_{F}^{-1}<1$.
If $\left|I\right|>K$ then already $\left|c_{I}\right|_{F}\leq q_{F}^{-1}$,
and the second rescaling cannot increase the absolute value. Hence
every coefficient of $\widetilde{g}$, except the coefficient of $x_{n}^{K}$,
has absolute value $<1$. Writing
\[
\widetilde{g}=\sum_{s=0}^{\infty}a_{s}(x_{1},\ldots,x_{n-1})x_{n}^{s},
\]
we see that $a_{K}$ has constant term of absolute value $1$, while
all its other coefficients have absolute value $<1$. Therefore $a_{K}$
is a unit in $T_{n-1}$. Also $\left\Vert \widetilde{g}\right\Vert _{\mathrm{Gauss}}=\left\Vert a_{K}\right\Vert _{\mathrm{Gauss}}=1$,
and $\left\Vert a_{s}\right\Vert _{\mathrm{Gauss}}<1$ for every $s>K$.
Thus $\widetilde{g}$ is distinguished in $x_{n}$ of order $K$.

Since the above operations consist only of shrinking the neighborhood
of 0, composing with analytic diffeomorphisms, and multiplying the
function by a non-zero scalar, they do not change the local integrability
exponent at $0$, so $\lct_{F}(g;0)=\lct_{F}(\widetilde{g};0)$.

By Theorem \ref{thm:Weierstrass preparation}, as in the proof of
Corollary \ref{cor:reduction to Weierstrass polynomial}, there exists
a Weierstrass polynomial $\widehat{g}=x_{n}^{K}+\sum_{i=0}^{K-1}r_{i}(x_{1},\ldots,x_{n-1})x_{n}^{i}$
of degree $K$, such that $\lct_{F}(g;0)=\lct_{F}(\widehat{g};0)$.
By the proof of Theorem \ref{thm:lower bound on lct}, we get that
$\lct_{F}(\widehat{g};0)\geq\frac{1}{K}$, as required. This concludes
the proof of the lemma. 
\end{proof}
\begin{proof}[Proof of Theorem \ref{thm:concrete lower bound}]
Let $x_{0}\in X(F)$. Without loss of generality, we may assume that
$\varphi_{1}$ doesn't vanish identically on the irreducible component
containing $x_{0}$. It is enough to prove that $\lct_{F}(\varphi_{1};x_{0})\geq\frac{1}{d\cdot D^{m}}$.

By Lemma \ref{lem:bounds on multiplicity}, $\mathrm{ord}_{X,x_{0}}(\varphi_{1})\leq d\cdot D^{m}$.
Changing coordinates to $\mathcal{O}_{F}^{n}$ via a diffeomorphism
$\psi:U\rightarrow\mathcal{O}_{F}^{n}$ with $\psi(x_{0})=0$ near
$x_{0}\in U\subseteq X$, and writing $f:=\varphi_{1}|_{U}\circ\psi^{-1}$,
it is enough to prove that $\lct_{F}(f;0)\geq\frac{1}{d\cdot D^{m}}$.
Since the multiplicity of $f$ at $0$ is bounded by $d\cdot D^{m}$,
Lemma \ref{lem:lct of low multiplicity} implies that $\lct_{F}(f;0)\geq\frac{1}{d\cdot D^{m}}$,
as required. 
\end{proof}
\begin{acknowledgement*}
We thank Rami Aizenbud, Raf Cluckers and Floris Vermeulen for useful
discussions. We further thank the anonymous referees for their useful
comments. I.G.~was supported by ISF grant 3422/24.\bibliographystyle{alpha}
\bibliography{bibfile}
\end{acknowledgement*}

\end{document}